\documentclass[12pt]{amsart}
\usepackage{amscd,amsmath,amsthm,amssymb,graphics}
\usepackage{pstcol,pst-plot,pst-3d}
\usepackage{lmodern,pst-node}
\usepackage[dvips]{graphicx}
\usepackage{multicol}
\usepackage{epic,eepic}
\usepackage{amsfonts,amssymb,amscd,amsmath,enumitem,verbatim}
\psset{unit=0.7cm,linewidth=0.8pt,arrowsize=2.5pt 4}

\newpsstyle{fatline}{linewidth=1.5pt}
\newpsstyle{fyp}{fillstyle=solid,fillcolor=verylight}
\definecolor{verylight}{gray}{0.97}
\definecolor{light}{gray}{0.9}
\definecolor{medium}{gray}{0.85}
\definecolor{dark}{gray}{0.6}

\def\frk{\frak}               

\def\Phi{{\frk n}}
\def\Phi{{\frk N}}
\def\Cc{{\mathcal C}}

\def\opn#1#2{\def#1{\operatorname{#2}}} 
\opn\chara{char} \opn\length{\ell} \opn\pd{pd} \opn\rk{rk}
\opn\projdim{proj\,dim} \opn\injdim{inj\,dim} \opn\rank{rank}
\opn\depth{depth} \opn\grade{grade} \opn\height{height}
\opn\embdim{emb\,dim} \opn\codim{codim}

\opn\Tr{Tr} \opn\bigrank{big\,rank}
\opn\superheight{superheight}\opn\lcm{lcm}
\opn\trdeg{tr\,deg}
\opn\reg{reg} \opn\lreg{lreg} \opn\ini{in} \opn\lpd{lpd}
\opn\size{size}\opn\bigsize{bigsize}
\opn\cosize{cosize}\opn\bigcosize{bigcosize}
\opn\sdepth{sdepth}\opn\sreg{sreg}
\opn\link{link}\opn\fdepth{fdepth}
\opn\div{div} \opn\Div{Div} \opn\cl{cl} \opn\Cl{Cl}
\let\epsilon\varepsilon
\let\phi=\varphi
\let\kappa=\varkappa
\opn\Spec{Spec} \opn\Supp{Supp} \opn\supp{supp} \opn\Sing{Sing}
\opn\Ass{Ass} \opn\Min{Min}\opn\Mon{Mon} \opn\dstab{dstab} \opn\astab{astab}
\opn\Syz{Syz}
\opn\Ann{Ann} \opn\Rad{Rad} \opn\Soc{Soc}
\opn\Im{Im} \opn\Ker{Ker} \opn\Coker{Coker} \opn\Am{Am}
\opn\Hom{Hom} \opn\Tor{Tor} \opn\Ext{Ext} \opn\End{End}
\opn\Aut{Aut} \opn\id{id}

\opn\nat{nat}
\opn\pff{pf}
\opn\Pf{Pf} \opn\GL{GL} \opn\SL{SL} \opn\mod{mod} \opn\ord{ord}
\opn\Gin{Gin} \opn\Hilb{Hilb}\opn\sort{sort}
\opn\initial{init}
\opn\ende{end}
\opn\height{height}
\opn\type{type}
\opn\aff{aff} \opn\con{conv} \opn\relint{relint} \opn\st{st}
\opn\lk{lk} \opn\cn{cn} \opn\core{core} \opn\vol{vol}
\opn\link{link} \opn\star{star}\opn\lex{lex}
\opn\gr{gr}

\def\pot#1#2{#1[\kern-0.28ex[#2]\kern-0.28ex]}

\opn\dirlim{\underrightarrow{\lim}}
\opn\inivlim{\underleftarrow{\lim}}
\let\tensor=\otimes

\let\to=\rightarrow

\def\Implies{\ifmmode\Longrightarrow \else
        \unskip${}\Longrightarrow{}$\ignorespaces\fi}
\def\implies{\ifmmode\Rightarrow \else
        \unskip${}\Rightarrow{}$\ignorespaces\fi}
\def\iff{\ifmmode\Longleftrightarrow \else
        \unskip${}\Longleftrightarrow{}$\ignorespaces\fi}

\let\:=\colon
 \theoremstyle{plain}
\newtheorem{Theorem}{Theorem}[section]
 \newtheorem{Lemma}[Theorem]{Lemma}

 \theoremstyle{definition}
 
 \newtheorem{Remark}[Theorem]{Remark}

\let\epsilon\varepsilon
\let\kappa=\varkappa
\opn\dis{dis}
\def\pnt{{\raise0.5mm\hbox{\large\bf.}}}

\opn\Lex{Lex}

\begin{document}
\title{On the relations of isotonian algebras}
\author {J\"urgen Herzog,  Ayesha Asloob Qureshi and  Akihiro Shikama}

\address{J\"urgen Herzog, Fachbereich Mathematik, Universit\"at Duisburg-Essen, Fakult\"at f\"ur Mathematik, 45117 Essen, Germany} \email{juergen.herzog@uni-essen.de}

\address{Ayesha Asloob Qureshi, Sabanc\i \; University, Orta Mahalle, Tuzla 34956, Istanbul, Turkey}
\email{aqureshi@sabanciuniv.edu}

\address{Akihiro Shikama, Department of Pure and Applied Mathematics, Graduate School of Information Science and Technology,
Osaka University, Toyonaka, Osaka 560-0043, Japan}
\email{a-shikama@cr.math.sci.osaka-u.ac.jp}

\thanks{This paper was partially written during the stay of the second author at The Abdus Salam International Centre of Theoretical Physics (ICTP), Trieste, Italy}

\begin{abstract}
It is shown that for large classes of posets $P$ and $Q$, the defining ideal $J_{P,Q}$ of an isotonian algebras is generated by squarefree binomials. Within these classes, those posets are classified for which $J_{P,Q}$ is quadratically generated.
\end{abstract}
\thanks{}

\maketitle
\section*{Introduction}

In \cite{BHHSQ}, isotonian algebras were introduced as a generalization of Hibi rings. Given two  finite posets $P$ and $Q$, the algebra $K[P,Q]$ is generated by monomials
\[
u_{\varphi}=\prod_{p\in P}x_{p,\varphi(p)}\quad \text{with}\quad p\in P,
\]
where $\phi: P \rightarrow Q$ is an isotone (order preserving) map. We denote the defining ideal of $K[P,Q]$ by $J_{P,Q}$. Since $K[P,Q]$ is a toric ring, the ideal $J_{P,Q}$ is generated by binomials.

Isotone maps first appeared in Commutative Algebra in the study of the ideals $I(P,Q)$ which are generated by all the monomials $u_{\phi}$, see \cite{FGH}. The structure of these ideals has been further studied in subsequent papers \cite{KKM}, \cite{HSQ}, \cite{FN},\cite{DFN1}, \cite{DFN2}.

In the case that $Q$ is a chain of length $2$, the algebra $K[P,Q]$ coincides with the so-called and well studied Hibi ring, see \cite{Hi}, \cite{HHBook}, \cite{HEM}, \cite{EHM}. It is known that the defining ideals of  Hibi rings admit  quadratic Gr\"obner basis and they are normal Cohen-Macaulay domains. In the above mentioned paper \cite{BHHSQ}, it is also shown that $J_{P,Q}$ admits a quadratic Gr\"obner basis if $P$ is a chain and $Q$ is a co-rooted poset.  However, easy examples show that isotonian algebras need not to admit quadratic Gr\"obner bases.  On the other hand, it is conjectured that the ideal $J_{P,Q}$ has a squarefree initial ideal for a suitable term order.

We call $f=u-v$ squarefree, if $u$ and $v$ are squarefree monomials.  In support of the above conjecture we prove in Theorem~\ref{sqfr}, that $J_{P,Q}$ is generated by squarefree binomials in the case that each connected component of $P$ is rooted or co-rooted, and that $J_{P,Q}$ is generated by quadratic binomials if and only if every poset cycle  of length $\geq 6$ in $Q$ admits a chord.

\section{Reduction to the case that  $P$ is connected}

Let $P$ be a finite poset. We say that $P$ is connected if its Hasse diagram is connected as a graph. Given two posets $P_1$ and $P_2$, the sum $P_1+P_2$ is defined to be the disjoint union of the elements  of $P_1$ and $P_2$ with $p\leq q$ if and ony if $p,q\in P_1$ or $p,q\in P_2$ and $p\leq q$ in the corresponding posets $P_1$ or $P_2$. Then its is clear that $P$ can be written as  $P=P_1+P_2+\cdots +P_r$ where each $P_i$ is connected. The $P_i$ are called the connected components of $P$.

Let $P,P_1,P_2,\ldots,P_r$ and $Q,Q_1,Q_2,\ldots,Q_s$ be finite posets. Then it is shown in \cite[Lemma 1.3]{BHHSQ} that  $K[\sum_{i=1}^rP_i,Q]$ is isomorphic to the Segre product $K[P_1,Q]*K[P_2,Q]*\cdots *K[P_r,Q]$. Recall that if  $R_1=K[f_1,\ldots,f_r]$ and $R_2=K[g_1,\ldots,g_s]$ are two standard graded $K$-algebras. Then

\[
R_1*R_2=K[\{f_ig_j\:\; i=1,\ldots,r,j=1,\ldots,s\}]
\]
is the Segre product of $R_1$ and $R_2$.

\medskip
We denote by $\Hom(P,Q)$ the set of isotone maps from $P \rightarrow Q$, and let $T=K[t_{\varphi}\: \; \varphi \in \Hom(P,Q)]$ be the polynomial ring in the  variables $t_{\varphi}$, and  let $\phi : T \rightarrow K[P,Q]$ be the $K$-algebra homomorphism defined by  $t_{\varphi} \mapsto \prod_{p\in P} x_{p\varphi(p)}$. We denote the kernel of  $\varphi$ by $I_{P,Q}$.

\medskip

\begin{Theorem}
\label{connected}
Let $P$ and $Q$ be two finite posets and let $P_1,P_2,\ldots, P_r$ be the connected components of $P$. Then $I_{P,Q}$ is generated by squarefree, respectively quadratic binomials,  if each $I_{P_i,Q}$ is generated by squarefree, respectively quadratic binomials.
\end{Theorem}

\begin{proof}

By using induction on $r$ it suffices to prove the following fact: Let $R_1=K[f_1, \ldots, f_r]$ and $R_2=K[g_1, \ldots, g_s]$ be standard graded toric rings whose defining ideals are generated by squarefree binomials, then this is also the case for $R_1* R_2$.
To show this we introduce some notation: let $T_1=K[x_1, \ldots, x_r]$ and $T_2=K[y_1, \ldots, y_s]$ be the polynomial ring over $K$, and $I_1$ the kernel of $K$-algebra homomorphism $T_1 \rightarrow R_1$ with $x_i \mapsto f_i$. The ideal $I_2$ is similarly defined. Next, let $T=K[z_{ij}\; | \; i= 1, \ldots, r, j=1, \ldots, s]$ be the polynomial ring over $K$ and $I$ the kernel of the $K$-algebra homomorphism $T \rightarrow R_1 * R_2$ with $z_{ij} \mapsto f_i g_j$.

We have to show that any binomials  $h=\prod_{t=1}^d z_{i_t j_t} - \prod_{t=1}^d z_{k_t l_t} \in I$ is a linear combination of squarefree, respectively quadratic binomials  of $I$. Note that $h= h_1 + h_2$, where
\[h_1=\prod_{t=1}^d z_{i_t j_t} - \prod_{t=1}^d z_{k_t j_t} ,  \quad h_2=\prod_{t=1}^d z_{k_t j_t} - \prod_{t=1}^d z_{k_t l_t}.
\]

Observe that $h_1, h_2 \in I$. Indeed, $h_1 \mapsto \prod_{t=1}^d( f_{i_t} g_{j_t}) - \prod_{t=1}^d (f_{k_t} g_{j_t} ) $, and we have to show that $\prod_{t=1}^d( f_{i_t} g_{j_t}) - \prod_{t=1}^d (f_{k_t} g_{j_t} ) =0$.

Since $h \in I$, it follows that $\prod_{t=1}^d( f_{i_t} g_{j_t}) - \prod_{t=1}^d (f_{k_t} g_{l_t} )=0$. Let $R_2 \rightarrow K$ be the $K$-algebra homomorphism such that $g_i \mapsto 1$. This $K$-algebra homomorphism induces a canonical $K$-algebra homomorphism $R_1 * R_2 \rightarrow R_1 *k=R_1$. This shows $\prod_{t=1}^d( f_{i_t} g_{j_t}) - \prod_{t=1}^d (f_{k_t} g_{l_t} )= 0$ is mapped to $ \prod_{t=1}^d f_{i_t} - \prod_{t=1}^d f_{k_t} =0 $ in $R_1$. Similarly we see that $ \prod_{k=1}^d g_{i_s} - \prod_{k=1}^d g_{j_k} =0 $ in $R_2$.

Next, let $R_1 * R_2 \rightarrow R_1 \tensor R_2= K[f_1, \ldots f_r, g_1, \ldots, g_s]$ be the canonical inclusion map. Then, $\prod_{t=1}^d( f_{i_t} g_{j_t}) - \prod_{t=1}^d (f_{k_t} g_{j_t} ) \mapsto (\prod_{t=1}^d f_{i_t} - \prod_{t=1}^d f_{k_t} )\prod_{t=1}^d g_{j_t} =0$. Therefore,
\[
\prod_{t=1}^d( f_{i_t} g_{j_t}) - \prod_{t=1}^d (f_{k_t} g_{j_t} )=0.
\]
Similarly,
\[
\prod_{t=1}^d( f_{k_t} g_{j_t}) - \prod_{k=1}^d (f_{k_t} g_{l_t} )=0.
\]
This shows that $h_1, h_2 \in I$.

Since $w=\prod_{i=1}^d x_{i_t} - \prod_{i=1}^d x_{k_t} \in I_1$, our assumption implies that $w$  can be written as linear combination of squarefree, respectively quadratic binomials $w_q$.  Then  $w=\sum_qu_qw_q$ where $u_q$ is a monomial in  $T_1$,  and $w_q \in I_1$ is squarefree, respectively quadratic.  
Let $g=\prod_{t=1}^d y_{j_t}$, and for each $q$ write $g=g_{1,q}g_{2,q}$,  where $g_{2,q}$  is the product of $\deg w_q$ many factors of $g$ and $g_{1,q}$ is the product of the remaining factors. Then $w=\sum_qu_qw_q$  yields the equation
\begin{eqnarray}
\label{eq}
wg=\sum_q(u_qg_{1,q})(w_qg_{2,q}),
\end{eqnarray}
which is valid in $R_1*R_2$.

Let us first assume that all $w_q$ are squarefree binomials. Then for each $q$,  we have    $w_q=\prod_{\in V_q}x_{i}-\prod_{i\in V_q'}x_{i}$ with $V_q,V_q'\subset \{1,\ldots,r\}$  and $q_{2,q}=\prod_{j\in U_q}y_j$ with $U_q\subset \{j_1,\ldots,j_d\}$ and $|U_q|=|V_q|=|V_q'|$. For each $q$, let $\sigma_q: V_q\to U_q$ and  $\tau_q: V_q\to U_q$ be bijections.
Then  (\ref{eq}) implies that  $h_1$ is a linear combination of  squarefree binomials of the form
\[
\prod_{i \in V_q} z_{i\sigma_q(i)} -  \prod_{i \in V'_q} z_{i\tau_q(i)}.
\]

The similar statement holds for $h_2$. This shows that $h$ is a linear combination of squarefree binomials. Similarly, if we assume all $w_q$ are quadratic binomials, we deduce in the same way that $h$ is a linear combination of quadratic binomials.
\end{proof}

\section{Special binomials}

In this section we study the relations of $K[P,Q]$ in the case that $P$ is rooted. Analogous arguments hold when $P$ is  co-rooted. 

To define special binomial relations, we first prove the following lemma in which a new isotone map is constructed from a given one by suitable modifications.



\begin{Lemma}
\label{map}
Let $P$  and $Q$ be two posets and assume that $P$ is rooted. Let $\phi \in \Hom(P,Q)$ and let  $\psi: P \rightarrow Q$ be a map  satisfying the following condition: there exists  $p \in P$ such that
 \begin{enumerate}
 \item[({\em i})] $\psi(p_1) \geq \psi(p_2)$, for all $p_1 \geq p_2\geq p$,
 \item[{\em (ii)}] $\psi(q) = \phi (q)$ for all $q \ngeq p$,
 \item[{\em (iii)}] if there exists a lower neighbor $q$ of $p$ then $\psi(p) \geq \psi(q)$.
 \end{enumerate}
 Then $\psi \in \Hom (P,Q)$.
\end{Lemma}
\begin{proof}
Let $p_1,p_2 \in P$ such that $p_1 \geq p_2$. We need to show that $\psi(p_1) \geq \psi(p_2)$. If $p_1 \geq p_2 \geq p$ or  $p_1, p_2 \ngeq p$, then by using (i) and (ii), we are done.  The only non trivial case is when $p_1 \geq p$ and $p_2 \ngeq p$.  Since $p_1 \geq p, p_2$ and $P$ is rooted it follows that $p$ and $p_2$ are comparable. Since $p_2 \ngeq p$, we must have $p_2 < p$. In particular, $p$ admits a lower neighbor $q$. Then by using (i), (ii) and (iii) we get
\[
\psi(p_1) \geq \psi(p) \geq \psi(q) = \phi(q) \geq \phi(p_2) = \psi(p_2).
\]
\end{proof}

Now we are ready to define special relations. Let $\phi_1, \ldots, \phi_d \in \Hom(P,Q)$, $p \in P$ and $\pi$ be a permutation on the set $\{1, \ldots, d\}$. In the case that $p$ admits a lower neighbor $p' \in P$ then we require that $\pi$ satisfies $\phi_{\pi(i)}(p) \geq \phi_i(p') $ for all $i$.

Then,  by Lemma~\ref{map},  the following maps
  \[
\phi'_i(q)= \left\{ \begin{array}{ll}
         \phi_{\pi(i)} (q), &  \text{if  $q \geq p$}, \\
         \phi_i (q), &\text{otherwise}.
                  \end{array} \right.
\]
are isotone. Consider the binomial $f = \prod_{k=1}^d  t_{\phi_i} - \prod_{k=1}^d   t_{\phi'_i}$. Then $f \in J_{P,Q}$, and $f \neq 0$ if $\pi$ is not identity. More generally, if $g= \prod_{k=1}^d  t_{\phi_i} - \prod_{k=1}^d   t_{\psi_i}$, Then $g \in J_{P,Q}$ if and only if  for all $p \in P$

\begin{equation}
\label{mult}
\{\varphi_1(p), \varphi_2(p), \ldots, \varphi_d(p) \} = \{\psi_1(p), \psi_2(p), \ldots, \psi_d(p) \}
\end{equation}
as multi-set.

The binomial $f = \prod_{k=1}^d  t_{\phi_i} - \prod_{k=1}^d   t_{\phi'_i}$ is  called {\em special} of type $(p, \pi)$ with respect to $\phi_1, \ldots, \phi_d \in \Hom(P,Q)$.
Note that $f$ is also special of type $(p, \pi^{-1})$ with respect to $\psi_1, \ldots, \psi_d$.

Now we are ready to prove the following

\begin{Theorem}
\label{special}
Let $P$ be a rooted or co-rooted poset. Then for any $Q$, the ideal $J_{P,Q}$ is generated by special binomials.
\end{Theorem}

\begin{proof}
Let $P$ be a rooted poset and $p_0$ be the minimal element of $P$. 
Also let

\[
f=t_{\varphi_1}t_{\varphi_2}\cdots t_{\varphi_d} - t_{\psi_1}t_{\psi_2}\cdots t_{\psi_d}
\]
 be a nonzero binomial in $J_{P,Q}$. After relabelling of $\psi_i$, we may assume that $\phi_i (p_0) = \psi_i(p_0)$ for all $i$. Let $\mathcal{I} \subset P$ be the largest poset ideal with the property that $\phi_i (p) = \psi_i(p)$ for all $p \in \mathcal{I}$. Since $f \neq 0$, it implies that $\mathcal{I} \neq P$.

Let $p' \in P \setminus \mathcal{I}$ such that  $\mathcal{I'}= \mathcal{I} \cup \{p'\}$ is again the poset ideal in $P$. Note that if $p'$ has a lower neighbor then it belongs to $\mathcal{I}$.



 We will show that $f=g+h$ where $g$ is a special binomial and

 \begin{equation}
 \label{h}
 h= \prod_{i=1}^d t_{\alpha_i} -  \prod_{i=1}^d  t_{\beta_i} \in J_{P,Q} \text { such that  $\alpha_i (p) = \beta_i(p)$ for all $p \in \mathcal{I'}$}.
 \end{equation}

 Since $|\mathcal{I'}| > | \mathcal{I}|$, the desired conclusion follows by induction.


Since
 \begin{equation*}
\{\varphi_1(p'), \varphi_2(p'), \ldots, \varphi_d(p') \} = \{\psi_1(p'), \psi_2(p'), \ldots, \psi_d(p') \}
\end{equation*}
there exists a permutation $\pi$ such that $\psi_i(p') = \phi_{\pi(i)} (p')$. Since we choose $\mathcal{I}$ to be maximal, it follows that  $\psi_i (p') \neq \phi_i(p')$ for some $i$. Hence, we see that $\pi$ is nontrivial. Furthermore, if $p'$ has a lower neighbor $p''$, then $p'' \in \mathcal{I}$.  Therefore,
 \[
\phi_{\pi(i)} (p') = \psi_i (p') \geq \psi_i (p'') = \phi_i (p'').
 \]

Hence,  $g= \prod_{i=1}^d t_{\phi_i}-  \prod_{i=1}^d t_{\phi'_i} \in J_{P,Q}$ is a special where,

\begin{equation}\label{phi}
\phi'_i(p)= \left\{ \begin{array}{ll}
         \phi_{\pi(i)} (p), &  \text{if  $p \geq p'$}, \\
         \phi_i (p), &\text{otherwise}.
                  \end{array} \right.
\end{equation}

Let $h= \prod_{i=1}^d t_{\phi'_i}-  \prod_{i=1}^d t_{\psi_i}$, then $f=g+h$. Since $f,g \in J_{P,Q}$, it implies that $h \in J_{P,Q}$.  Moreover, $h$ satisfies (\ref{h}). This proves the theorem.
\end{proof}

\section{Squarefree binomials}

In this section we will prove that $J_{P,Q}$ is generated by squarefree binomials if each connected component of $P$ is rooted or co-rooted poset. In general, without any restriction on $P$ and $Q$, we have

\begin{Remark}
\label{quad}
Let $P$ and $Q$ be two posets. Then all quadratic binomials in $J_{P,Q}$ are squarefree. Indeed, assume that  $f=t_{\varphi_1} ^2- t_{\psi_1}t_{\psi_2}$.  Then $ \{\phi_1(p), \phi_1(p) \} = \psi_1(p), \psi_2(p) \} $, for all $p \in P$. This implies that $\psi_1 = \psi_2= \phi_1$ and hence $f=0$. This implies that any non-trivial quadratic binomial in $J_{P,Q}$ is squarefree.
\end{Remark}

We first prove the following

\begin{Lemma}\label{cyclic}
Any special binomial of type $(p, \pi)$ can be written as a linear combination of special binomials of type $(p, \pi')$ where the $\pi'$ are cyclic permutations.
\end{Lemma}

\begin{proof}
Let $\pi = \pi_1 \pi_2 \ldots \pi_l$ be the unique cycle decomposition of $\pi$.Then $f= \prod_{i=1}^l f_i^+ - \prod_{i=1}^l f_i^- $ where each $f_i = f_i^+ - f_i^-$ is a special binomial of type $(p, \pi_i)$. From this presentation of $f$, it can be shown by induction on $l$ that $f$  can be written as a linear combination of the $f_i$.
\end{proof}

Now we prove the main theorem.

\begin{Theorem}
\label{sqfr}
Let $P$ be a poset whose connected components are rooted or co-rooted posets. Then for any $Q$, the ideal $J_{P,Q}$ is generated by squarefree special binomials.
\end{Theorem}

\begin{proof}
By using Theorem~\ref{connected}, we may assume that $P$ is connected. Furthermore, we may also assume that $P$ is rooted. We will show by induction on degree of $f$ that $f$ is a linear combination of squarefree binomials in $J_{P,Q}$.

Let $f=t_{\varphi_1}t_{\varphi_2}\cdots t_{\varphi_d} - t_{\psi_1}t_{\psi_2}\cdots t_{\psi_d} \in J_{P,Q}$ and $f \neq 0$. By Theorem~\ref{special} and the Lemma~\ref{cyclic}, we may assume that $f$ is a special binomial  of type $(p, \pi)$ where $\pi$ is the cyclic permutation given by $(1\; 2 \; \ldots \; d)$. By Remark~\ref{quad}, we may assume that $d>2$. If all $\phi_i$ as well as all $\psi_i$ are pairwise distinct then $f$ is squarefree and there is nothing to prove. Assume that $\phi_i= \phi_j $. 

 Let $g= g^+ - g^-$ where $g^+=  t_{\phi_i}  t_{\phi_{i+1}} \ldots t_{\phi_{j-1}}$ and $g^-= t_{\psi_i}  t_{\psi_{i+1}} \ldots t_{\psi_{j-1}}$. Then $g \in J_{P,Q}$ because it is special of type $(p, \pi')$ with respect to $\phi_i, \phi_{i+1}, \ldots, \phi_{j-1}$ and $\pi' = (i \; i+1 \; \cdots \; j-1)$. Indeed, if $p' \ngeq p$ then $\psi_r(p') = \phi_r (p')$ , and if $p' \geq p$ then

 \begin{equation*}
\psi_r(p')=  \phi_{r+1} (p')  \text{ for  $r=i, \ldots, j-1$ and } \phi_j (p') = \phi_i (p')
\end{equation*}


Hence, we may write
\[
f = \frac{f ^+}{g^+ }\;g  +  g^-\; h
\]
where
\[
h=\frac{f ^+}{g^+ } - \frac{f ^-}{g^- }
\]
Since $\deg g, \deg h < \deg f$, by induction we know that $g$ and $h$ can be written as a linear combination of squarefree binomials. This completes the proof.
\end{proof}

\section{Quadratic binomials}

Here we will prove that $J_{P,Q}$ is generated by quadratic squarefree binomials if each connected components of $P$ is rooted or co-rooted and $Q$ does not contain any proper poset cycle of length $\geq 6$.

Let $Q$ be a poset.  A sequence $q_1, q_2, \ldots, q_{2m}$ of elements in $Q$ is called a {\em poset cycle} of length $2m$ if $q_1 \leq  q_2 \geq q_3 \leq \ldots \leq q_{2m}\geq q_1$.

We say that a  poset cycle of length $\geq 6$ has a chord if  there exist odd  $i$ and even $j$ with $j\neq i-1,i+1$ such that $q_i \leq q_j$. (Here $0$ is identified with $2m$).

A poset cycle is called {\em proper} if all elements $q_i$ in the cycle are pairwise distinct and the element with even indexes are pairwise incomparable as well as elements with odd indexes.

\begin{figure}[htbp]
\includegraphics[width = 5cm]{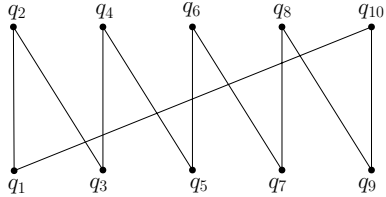}
\caption{A proper poset cycle of length 10}\label{something2}
\end{figure}

\begin{Lemma}\label{chord}
Let the sequence  $\Cc: q_1, q_2, \ldots, q_{2m}$ of length $\geq 6$ be a poset cycle in $Q$. If $\Cc$ is not a proper poset cycle then it admits a chord.
\end{Lemma}

\begin{proof}
If $\Cc$ is not a proper poset cycle then we have one of the following cases:

 \begin{enumerate}
 \item[(i)]$q_i \leq q_j$ where both $i$ and $j$ are even or odd.
 \item[(ii)] $q_i = q_j$ for some odd $i$ and even $j$.
 \end{enumerate}

In case (i), if $i$ and $j$ are even then we have $q_{i-1} \leq q_i \leq q_j$. Then the relation $q_{i-1}\leq q_j$ gives us a chord. Similarly, if $i$ and $j$ are odd then $ q_i \leq q_j \leq q_{j+1}$ and $q_i \leq q_{j+1}$ gives us the chord.

Now we discuss the case (ii) and assume first that $i<j$. Then $q_{j-1} \leq q_j = q_i$ where $j-1$ and $i$ are both odd. If $j-1\neq i$, then we are in case (i). Assume now that $j-1=i$. Then $q_{j+1}\leq q_{j-1}$, and we are again in case (i).

Similar arguments holds when $i>j$.
\end{proof}

Now we are ready to prove the main Theorem of this section.

\begin{Theorem}
Let $P$ and $Q$ be two posets.
\begin{itemize}
\item[{\em (a)}] If $P$ is an antichain then $J_{P,Q}$ is generated by quadratic squarefree binomials.
\item[{\em (b)}] If $P$ is not an antichain and  each connected component of $P$ is rooted or co-rooted, then the following conditions are equivalent:
\item[{\em (i)}] $J_{P,Q}$ is generated by quadratic  binomials.
\item[{\em (ii)}] $J_{P,Q}$ is generated by quadratic squarefree binomials.
\item[{\em (iii)}]  Each proper poset cycle of $Q$ of length $\geq 6$ has a chord.
\end{itemize}
\end{Theorem}

\begin{proof}
By Theorem~\ref{connected}, we may assume that $P$ is connected. Then in case $(a)$, the poset $P$ is just a single element and $K[P,Q]$ is a polynomial ring.

(b)  Since $P$ is connected, we may assume that $P$ is rooted.

(i) \iff (ii): The proof follows from Remark~\ref{quad}.

(i) \implies (iii): Let $\mathcal{C}: q_1, q_2, \ldots, q_{2m}$ be a proper poset cycle in $Q$  of length $2m$ with $m\geq 3$. Let $p_1$ be the minimal element of the rooted poset $P$. We define the following maps. For $i=1, \ldots, m$,

 \begin{equation*}
\phi_i(p)= \left\{ \begin{array}{ll}
         q_{2i-1}, &  \text{$p=p_1$}, \\
        q_{2i}, & \text{otherwise}.
                  \end{array} \right.
\end{equation*}
and

\begin{equation*}
\psi_1(p)= \left\{ \begin{array}{ll}
         q_{1}, &  \text{$p=p_1$}, \\
        q_{2m}, & \text{otherwise}.
                  \end{array} \right.
\end{equation*}
and for $i=2, \ldots, m-1$

 \begin{equation*}
\psi_i(p)= \left\{ \begin{array}{ll}
         q_{2i-1}, &  \text{$p=p_1$}, \\
        q_{2i-2}, & \text{otherwise,}
                  \end{array} \right.
\end{equation*}

Then $f=t_{\phi_1}\ldots t_{\phi_m} - t_{\psi_1}\ldots t_{\psi_m} \in J_{P,Q}$ is a special binomial of degree $m$ and type $(p, \pi)$ for some $p \in P$ with $p \neq p_1$. After relabeling of indexes, we may assume that $\pi = (1 \; m \; m-1 \; \ldots \; 2)$ and . Then, since by our assumption $J_{P,Q}$ is generated by quadratic binomials, there exists a nonzero quadratic binomial $g=g^+ - g^- \in J_{P,Q}$ such that $g^+ | f^+$. We may assume that $g = t_{\phi_i} t_{\phi_j} - t_{\gamma_1} t_{\gamma_2}$  for some $i< j $ and $\phi_i(p_1) = \gamma_1 (p_1)$ and $\phi_j (p_1) = \gamma_2 (p_1)$.  Moreover, we have

\[
\{ \phi_i(p) , \phi_j(p)\}= \{\gamma_1(p) , \gamma_2(p)\} \text{ for all $p \in P$}
\]

Since $g \neq 0$ there exists $p\in P$ and $p \neq p_1$ such that $\phi_i(p) = \gamma_2(p)$ and $\phi_j(p) = \gamma_1(p)$.  Therefore, $\phi_i(p) = q_{2i} \geq \phi_j(p_1) = q_{2j-1}$ and $\phi_{j} (p)= q_{2j} \geq \phi_i(p) = q_{2i-1}$. The sequence $\mathcal{C'} = q_{2i-1}, q_{2i}, q_{2j-1}, q_{2j}$ is a 4-cycle in $Q$. In particular $\mathcal{C}$ has a chord.




(iii) \implies (i): Let $f \in J_{P,Q}$  and $f \neq 0$. Then by Theorem~\ref{special}, we may assume that $f=t_{\phi_1}\ldots t_{\phi_m} - t_{\phi'_1}\ldots t_{\phi'_m} \in J_{P,Q}$ is  a special  binomial of type $(p,\pi)$ where $\pi$ is a cyclic permutation. After a relabeling of indxes of $\phi_i$ and $\phi'_i$, we may assume that $(1 \; m  \; m-1 \; \cdots \; 2)$. If $m=2$, then the assertion is trivial. Let $m \geq 3$. Since $f \neq 0$ there exists a lower neighbor of $p'$ of $p$ which is unique because $P$ is rooted. Then the sequence  $\mathcal{C}: q_1, q_2, \ldots, q_{2m}$ with  $ q_{2i-1} = \phi_i(p')$ and $q_{2i} = \phi_i(p)$ is a poset cycle in $Q$. 

By Lemma~\ref{chord} and (iii), we know that $\Cc$ admits a chord. We may assume that the chord is given by the relation $q_1 \leq q_{2i}$ with $i \neq m$.

Let $g= t_{\psi_1} \prod_{j=i+1}^m t_{\phi_j}- t_{\phi'_1}  \prod_{j=i+1}^m t_{\phi'_j}$ where

 \begin{equation*}
\psi_1(p'')= \left\{ \begin{array}{ll}
         \phi_{i} (p''), &  \text{$p'' \geq p$}, \\
        \phi_1(p''), & \text{otherwise}.
                  \end{array} \right.
\end{equation*}

By using Lemma~\ref{map}, it can be seen that $\psi_1$ is isotone. Observe that $g \in J_{P,Q}$. To see this, we have to show that for all $p'' \in P$

\begin{equation*}
\{\psi_1(p''),   \varphi_{i+1}(p''), \ldots, \varphi_m(p'') \} = \{\phi'_1(p''),  \varphi'_{i+1}(p''), \ldots, \phi'_m(p'') \}
\end{equation*}
as multi-set.

If $p''\geq p$, then $\psi_1(p'') = \phi_i(p'')$ and $\phi'_{j}(p'') = \phi_{j-1} (p'')$ for $j=i+1, \ldots, m$ and $\phi'_1(p'') = \phi_{m} (p'')$. In this case, the two multi-set coincides.

If $p'' \ngeq p$, then $\psi_1(p'') = \phi_1(p'')$ and $\phi'_j(p'') = \phi_{j} (p'')$ for $j=1,i+1, \ldots, m$. Again, the two multi-sets coincides.


Then

\[
f  -  t_{\phi'_2} \ldots t_{\phi'_{i}} g = \prod_{j=i+1}^m t_{\phi_j}  h ,
\]

where
\[
h=t_{\phi_1} t_{\phi_2} \ldots t_{\phi_{i}} - t_{\psi_1} t_{\phi'_2} \ldots t_{\phi'_{i}}.
\]

Since $f,g \in J_{P,Q}$ and $J_{P,Q}$ is a prime ideal, it follows that $h \in J_{P,Q}$. Note that $\deg g, \deg h < \deg f$. Then by induction, we know that $g$ and $h$ can be written as linear combination of quadratic binomials, which completes the argument.

\end{proof}


\begin{thebibliography}{}


\bibitem{BHHSQ} M.~ Bigdeli, J.~Herzog, T.~Hibi, A.~Shikama, A.~A.~Qureshi, Isotonian algebras.


\bibitem{FGH} G.~Fl{\o}ystad, B.M.~Greve and J.~Herzog, Letterplace and co-letterplace ideals of posets,

\bibitem{EHM}
V.~Ene, J.~Herzog and F.~Mohammadi,
Monomial ideals and toric rings of Hibi type arising from a finite poset,
 European J. Combin. {\bf 32} (2011), 404--421.


\bibitem{HHBook}
J.~Herzog and T.~Hibi, Monomial ideals,  Graduate Texts in Mathematics {\bf 260},
Springer, London, 2010.


\bibitem{HSQ} J.~Herzog, A.~A.~Qureshi and A.~Shikama,  Alexander duality for monomial ideals associated with isotone maps between posets, to appear in J. Algebra and its Applications.

\bibitem{Hi} T.~Hibi, Distributive lattices, affine semigroup rings and algebras with straightening
laws, Commutative Algebra and Combinatorics (M. Nagata and H. Matsumura, Eds.)
Adv. Stud. Pure Math. {\bf 11}, North-Holland, Amsterdam, 1987, 93--109.


\bibitem{HEM} V.~Ene, J.~Herzog and S.~ S.~ Madani, A note on the regularity of Hibi rings, Manuscripta Math., {\bf 148} (2015), 501--506.

\bibitem{KKM} M.~Juhnke-Kubitzke, L.~Katth\"an and  S.~S.~Madani, Algebraic properties of ideals of poset homomorphisms, 



\bibitem{FN} G.~Fl{\o}ystad, A.~Nematbakhsh,  Rigid ideals by deforming quadratic letterplace ideals,




\bibitem{DFN1} A.~D'Al\'{i},  G.~Fl{\o}ystad, A.~Nematbakhsh,  Resolutions of co-letterplace ideals and generalizations of Bier spheres,



\bibitem{DFN2} A.~D'Al\'{i}, G.~Fl{\o}ystad, A.~Nematbakhsh,  Resolutions of letterplace ideals of posets,


\end{thebibliography}
\end{document}